\documentclass[12pt]{article}
\usepackage[utf8]{inputenc}
\usepackage{geometry}
\usepackage{mathtools}
\usepackage{enumitem}
\usepackage{amsmath}
\usepackage{amsthm}
\usepackage{amssymb}
\usepackage{stmaryrd}
\usepackage{cases}
\usepackage[numbers]{natbib}
\usepackage[unicode=true,pdfusetitle,
 bookmarks=true,bookmarksnumbered=false,bookmarksopen=false,
 breaklinks=false,pdfborder={0 0 1},backref=false,colorlinks=false]
 {hyperref}

\makeatletter

\@ifundefined{date}{}{\date{}}
\usepackage{titlesec}
\usepackage[titletoc,toc,title]{appendix}
\usepackage{wasysym}
\usepackage{xcolor}

\usepackage{graphicx}

\theoremstyle{plain}
\newtheorem{theorem}{Theorem}
\newtheorem{lemma}[theorem]{Lemma}

\newtheorem{proposition}[theorem]{Proposition}

\theoremstyle{definition}
\newtheorem{remark}[theorem]{Remark}

\DeclareMathOperator{\E}{{\mathbb E}}

\DeclareMathOperator{\D}{{\mathbb D}}
\DeclareMathOperator{\R}{{\mathbb R}}

\DeclareMathOperator{\N}{{\mathbb N}}

\DeclareMathOperator{\Q}{{\mathbb Q}}

\DeclareMathOperator{\dom}{dom}

\renewcommand{\phi}{\varphi}

\renewcommand{\theta}{\vartheta}
\renewcommand{\subset}{\subseteq}

\providecommand{\abs}[1]{\lvert #1 \rvert}
\providecommand{\norm}[1]{\lVert #1 \rVert}
\providecommand{\bnorm}[1]{{\Bigl\lVert #1 \Bigr\rVert}}

\providecommand{\wraum}{$(\Omega,{\scr F},\PP)$}
\providecommand{\fwraum}{$(\Omega,{\scr F},\PP,({\scr F}_t))$}

\makeatother

\begin{document}
\global\long\def\epsilon{\varepsilon}

\global\long\def\E{\mathbb{E}}

\global\long\def\I{\mathbf{1}}

\global\long\def\N{\mathbb{N}}

\global\long\def\R{\mathbb{R}}

\global\long\def\C{\mathbb{C}}

\global\long\def\Q{\mathbb{Q}}

\global\long\def\P{\mathbb{P}}

\global\long\def\D{\Delta_{n}}

\global\long\def\dom{\operatorname{dom}}

\global\long\def\b#1{\mathbb{#1}}

\global\long\def\c#1{\mathcal{#1}}

\global\long\def\s#1{{\scriptstyle #1}}

\global\long\def\u#1#2{\underset{#2}{\underbrace{#1}}}

\global\long\def\r#1{\xrightarrow{#1}}

\global\long\def\mr#1{\mathrel{\raisebox{-2pt}{\ensuremath{\xrightarrow{#1}}}}}

\global\long\def\t#1{\left.#1\right|}

\global\long\def\l#1{\left.#1\right|}

\global\long\def\f#1{\lfloor#1\rfloor}

\global\long\def\sc#1#2{\langle#1,#2\rangle}

\global\long\def\abs#1{\lvert#1\rvert}

\global\long\def\bnorm#1{\Bigl\lVert#1\Bigr\rVert}

\global\long\def\wraum{(\Omega,\c F,\P)}

\global\long\def\fwraum{(\Omega,\c F,\P,(\c F_{t}))}

\global\long\def\norm#1{\lVert#1\rVert}

\global\long\def\var{{\bf var}}

\global\long\def\1{\mathbf{1}}

\global\long\def\theta{\vartheta}

\date{}

\author{
Randolf Altmeyer\thanks{Center for Mathematical Sciences, University of Cambridge, Wilberforce Road, CB3 0WB Cambridge, UK, Email: ra591@cam.ac.uk} \and  Ronan~Le Guével\thanks{Univ Rennes, CNRS, IRMAR - UMR 6625, F-35000 Rennes, France. Email: ronan.leguevel@univ-rennes2.fr}}

\title{Optimal $L^2$-approximation of occupation and local times for symmetric stable processes}
\maketitle
\begin{abstract}
The $L^2$-approximation of occupation and local times of a symmetric $\alpha$-stable L\'evy process from high frequency discrete time observations is studied. The standard Riemann sum estimators are shown to be asymptotically efficient when $0<\alpha\leq 1$, but only rate optimal for $1<\alpha\leq 2$. For this, the exact convergence of the $L^2$-approximation error is proven with explicit constants. 
\end{abstract}

\section{Introduction}

Let $X=(X_t)_{t\geq0}$ be a scalar stochastic process. Two path dependent functionals of $X$ which are of interest in many applications are its occupation and local times respectively defined by
\begin{align}
    \mathcal{O}_T(A) = \int_0^T \I_A(X_t)dt \quad \textrm{and} \quad L_T(y)= \frac{d\c O_T}{dy}(y),\label{eq:occLocTime}
\end{align}
which measure the time the process spends inside a Borel set $A\subset\R$ or at a point $y\in\R$, whenever the occupation measure $A\mapsto \mathcal{O}_T(A)$ is absolutely continuous with respect to the Lebesgue measure. We aim at studying optimal $L^2$-approximations of these functionals given the observations $X_{t_k}$ at $t_k=k\Delta_n$ for $k=1,\dots,n$ with time distance $\Delta_n=T/n$, where the time horizon $T>0$ is fixed and in the high frequency limit as $n\rightarrow\infty$. The minimal $L^2$-error is achieved for the conditional expectations $\E[\mathcal{O}_T(A)|\c G_n]$ and $\E[L_T(y)|\c G_n]$, where $\c G_n$ is the sigma field generated by the $X_{t_k}$, but these two conditional expectations may be unfeasible to compute when the law of $X$ is unknown. Instead, the standard estimators in the literature are based on integral approximations using the Riemann sums
\begin{align}
	\hat{\c O}_{T,n}(A) = \Delta_{n}\sum_{k=1}^{n}\I_{A}(X_{t_{k-1}}) \quad \textrm{and} \quad
	\hat{L}_{T,n}(y) = \frac{\Delta_n}{2 h_n} \sum_{k=1}^{n} \I_{[y-h_n,y+h_n]}(X_{t_{k-1}})\label{eq:riemannEstimators}
\end{align}
for some bandwidth parameter $h_n>0$. These approximations may be far from optimality since they crucially depend on the smoothness of the law of the underlying process. The main result of this article is to settle this question in the context of symmetric $\alpha$-stable processes by proving exact convergence results for the $L^2$-approximation errors.

The approximation of occupation and local times is important in many applications and has been extensively studied in the literature. 
For stationary continuous time stochastic processes for instance, the irregularity of the sample paths implies non-standard rates of convergence in the non-parametric estimation of the probability density with kernel type estimators, as has been noticed in \cite{CastellanaLeadbetter1986, Bosq1998}. The question of optimality with respect to the sampling of discrete time observation schemes has been studied in \cite{BlankePumo2003}, while the rate optimality is considered in \cite{ComteMerlevede2003} through the study of a projection estimator. We focus in this article on non-parametric methods, but for processes which are no longer stationary. For scalar diffusion processes $X$ and intervals $A$, the standard estimators have been studied by several authors \cite{borodinCharacterConvergenceBrownian1986, Ngo2011, Jacod1998, Kohatsu-Higa2014}, satisfying the rates of convergence $\Delta_n^{3/4}$ for $\mathcal{O}_T(A)$ and $\Delta_n^{1/4}$ for $L_T(y)$. These rates can be explained in the context of $L^2$-approximations of integral functionals $\int_0^T g(X_t)dt$ for non-smooth integrands $g$ \cite{Altmeyer2017}. In this way, \cite{Altmeyer2019c, altmeyerApproximationOccupationTime2021} obtain similar results for more general Markovian and non-Markovian processes such as semimartingales or fractional Brownian motion. Rate optimality of the Riemann sum estimators in the case of Brownian motion (with drift) can be obtained from  \cite{ Ngo2011, Ivanovs2020, altmeyerApproximationOccupationTime2021}, but it is unclear if their methods extend to jump processes, or if Riemann estimators are asymptotically efficient in the sense of reaching the minimal asymptotic error. More recently, there is also some interest in numerical analysis for the $L^p$-approximation error in the context of analysing Euler schemes with non-degenerate coefficients (\cite{neuenkirchEulerMaruyamaScheme2020}, \cite{muller2020sharp}), see also \cite{butkovsky2021approximation}.

Similar to \cite{Ivanovs2020, altmeyerApproximationOccupationTime2021}, we assess the question of optimality by studying the conditional expectations $\E[\mathcal{O}_T(A)|\c G_n]$, $\E[L_T(y)|\c G_n]$. For explicit computations, we restrict to symmetric $\alpha$-stable processes for $0<\alpha\leq 2$, but we expect that our results hold also for more general Lévy processes. We prove the exact constants for the asymptotic $L^2$-approximation errors of the conditional expectations and for the Riemann sum estimators. In both cases we obtain the rates of convergence $\Delta_n^{(1+\min(1,1/\alpha))/2}$ for occupation times (up to $\log$-factors) and $\Delta_n^{(1-1/\alpha)/2}$ for local times. In particular, we show that the Riemann sum estimators are rate optimal, but asymptotically efficient only for $\alpha \leq 1$, surprisingly. Let us point out that while the conditional expectations are explicit estimators up to the possibly unknown parameter $\alpha$, they depend on the marginal densities of $X$ and are therefore not known analytically for $\alpha<2$, requiring numerical approximations. Our results imply, however, that it is sufficient to use the standard estimators. The general proof strategy is to compute $L^2$ terms explicitly, leading to Riemann integrals and then argue by dominated convergence or improper integral divergence, using precise asymptotics with respect to the law of $\alpha$-stable distributions.

The paper is organised as follows. In Section \ref{sec:stable} we recall properties of stable processes. Section \ref{sec:riemann}  presents the $L^2$-approximation results for standard estimators of occupation and local times. Consistency and rates of convergence in a general setting are discussed, along with exact asymptotics for some important cases. Section \ref{sec:optimal} compares these results to the exact asymptotics for the conditional expectations. All proofs are deferred to Section \ref{sec:proofs}.

\section{Main results}

\subsection{$\alpha$-stable processes}\label{sec:stable}

Suppose that $X=(X_t)_{t\geq 0}$ is a scalar symmetric $\alpha$-stable L\'evy process for $0<\alpha\leq 2$, that is $X_0=0$ and $X$ is a self-similar Lévy process such that
\[
	(X_{bt})_{t\geq 0} \overset{d}{=}b^{1/\alpha}(X_{t})_{t\geq 0}, \quad b>0.
\]
In particular, each $X_t$ has for $t>0$ the characteristic function $u\mapsto\E[e^{iuX_t}]=e^{-|u|^{\alpha}t}$ and thus the Lebesgue density  
\begin{equation}\label{expressionautosim}
 f_{\alpha,t}(x)=\frac{1}{t^{1/\alpha}}{f_{\alpha}\left(\frac{x}{t^{1/\alpha}}\right)}\quad \textrm{with} \quad f_{\alpha}(x)  = \frac{1}{\pi} \int_0^{\infty} e^{-t^{\alpha}} \cos(xt)dt, \quad x \in \R,
\end{equation} 
cf. \cite{satoLevyProcessesInfinitely1999}. For $\alpha=2$, $X$ is a Brownian motion (up to scaling factor), and a Cauchy-process for $\alpha=1.$ 

Since $X$ has right-continuous paths, the occupation time $\mathcal{O}_T(A)$ in \eqref{eq:occLocTime} is well-defined for each Borel set $A\subset\R$ and any $0<\alpha\leq 2$. We write $$\c O_T(y)  = \c O_T([y,\infty)), \quad y \in \R.$$The local time process $L_T(y)$ in \eqref{eq:occLocTime}, however, exists only for $\alpha > 1$, cf. \cite[Theorem 2.1]{koshnevisan2003}. Recall also the occupation time formula 
\begin{equation}\label{occfor}
\int_0^T f(X_s)ds = \int_{\mathbb{R}} f(x) L_T(y) dy,
\end{equation}
which holds for all nonnegative measurable functions $f$, cf. \cite{samorodnitskyStochasticProcessesLong2016}. 

\subsection{Results for the Riemann estimators}\label{sec:riemann}

We begin by a simple, but general consistency result for $\hat{\mathcal{O}}_{T,n}(A)$, which can be shown exactly as in \cite[Proposition 2.1]{Ngo2011}. Convergence of $\hat{L}_{T,n}(y)$ to $L_T(y)$ in probability (and in $L^2$) will follow from Theorem \ref{theorem:Theo3} below. 

\begin{proposition}
Let $X$ be a stochastic process with right-continuous (or left-continuous) paths such that for all $t>0$ the law of $X_{t}$ is absolutely continuous with respect to
the Lebesgue measure $\lambda$. If  $A\subset \R$ is a Borel set such
that $\lambda(\bar{A}\backslash\mathring{A})=0$, where $\bar{A}$ and $\mathring{A}$ denote the closure and the interior of $A$, then $\P$-almost surely
\[
\lim_{n\rightarrow\infty}\hat{\mathcal{O}}_{T,n}(A)=\mathcal{O}_T(A).
\]
\end{proposition}

Let us consider $L^2$-rates of convergence of this estimator. For $\alpha>1$ we follow the proof strategy outlined in \cite[Corollaries 6 and 7]{altmeyerApproximationOccupationTime2021} for fractional Brownian motion by upper bounding the characteristic function of the bivariate distributions $(X_t,X_{t'})$ for $0<t<t'\leq T$. This leads to a general control of the error approximation on every interval $[a,b]$.

\begin{theorem}\label{theorem:OccTime}Let $1<\alpha\leq 2$ and let $-\infty \leq a \leq b\leq \infty$. Then for sufficiently small $\epsilon>0$ 
\[
	 \Delta_n^{-1-1/\alpha}\norm{\hat{\c O}_{T,n}([a,b])-\c O_T([a,b])}^2_{L^{2}(\P)} \leq C (1\vee T^{-\epsilon})T^{1-1/\alpha},
	\]
with $C<\infty$ is independent of $a,b,T$ and $n$. 
\end{theorem}

The proof of Theorem \ref{theorem:OccTime} breaks down for $0<\alpha\leq 1$ due to the singularity of $f_{\alpha,t}$ near $t=0$, yielding only suboptimal rates of convergence. This can be resolved by assuming an initial distribution $X_0$ having a bounded Lebesgue density and the same upper bound from Theorem \ref{theorem:OccTime} still applies up to a small polynomial loss in the rate of convergence, as has been shown in \cite[Theorem 15]{altmeyerApproximationOccupationTime2021}, noting that indicator functions of bounded intervals have fractional Sobolev regularity $s<1/2$. 

In the important case when $A=[0,\infty)$ is a half-line, we will now obtain exact convergence results with explicit constants for all $0<\alpha\leq 2$. This is new even in the Brownian case (upper and lower bounds in the this case can be found in \cite[Proposition 2.3]{Ngo2011}).

\begin{theorem}
\label{Theo1} 
Define $\psi(x) := (x-\lfloor x\rfloor)-(x-\lfloor x\rfloor)^{2}$, $x\geq 0$, where $\lfloor x\rfloor$ is the integer part of $x$. If $1<\alpha \leq 2$, then
\[
\lim_{n\rightarrow\infty}\Delta_{n}^{-1-1/\alpha}\norm{\hat{\c O}_{T,n}(0)-\c O_T(0)}_{L^{2}(\P)}^{2}=T^{1-1/\alpha}\frac{2 \Gamma(1/\alpha)}{\pi \alpha^2}\E[|X_1|]\int_{0}^{\infty}\frac{\psi(x)}{x^{2-1/\alpha}}dx.
\]
If $0<\alpha<1$, then
\[
 \lim_{n\rightarrow\infty}\Delta_{n}^{-2} (\log n)^{-1}\norm{\hat{\c O}_{T,n}(0)-\c O_T(0)}_{L^{2}(\P)}^{2} =\frac{\Gamma(\alpha)\sin(\frac{\pi \alpha}{2} )}{12 \pi} \E[|X_1|^{-\alpha} ].
\]
If $\alpha =1$, then
\[
\lim_{n\rightarrow\infty}\Delta_{n}^{-2} (\log n)^{-2}\norm{\hat{\c O}_{T,n}(0)-\c O_T(0)}_{L^{2}(\P)}^{2} = \frac{1}{12 \pi^2}.
\]
\end{theorem}

We will see in the next section that the additional $\log n$ factors for $0<\alpha\leq 1$ are necessary. 
Using Theorem \ref{theorem:OccTime} we also state a general upper bound for the estimation of the local time. 

\begin{theorem}
\label{theorem:Theo3}Let $1<\alpha\leq2$ and let $y\in\R$. Then for any sufficiently small $\epsilon>0$
\[
\norm{\hat{L}_{T,n}(y)-L_{T}(y)}^2_{L^{2}(\P)}\leq  C (1 \vee T^{-\epsilon}) T^{1-1/\alpha} (h_n^{\alpha-1} +  \Delta_{n}^{1+ 1/\alpha} h_n^{-2}),
\]
where $C< \infty$ is independent of $a,b,T$ and $n$. If $h_n=\Delta_n^{1/\alpha}$, then
\[ \Delta_n^{1/\alpha-1}\norm{\hat{L}_{T,n}(y)-L_{T}(y)}^2_{L^{2}(\P)}\leq C (1 \vee T^{-\epsilon}) T^{1-1/\alpha}.  \]

\end{theorem}

In the Brownian case we recover the rate of convergence $\Delta_n^{1/4}$ from \cite{Jacod1998}, and therefore improve on \cite[Corollary 7]{altmeyerApproximationOccupationTime2021} and \cite[Theorem 2.6]{Kohatsu-Higa2014}.

\subsection{Optimal estimation results}\label{sec:optimal}

In this section we will derive the exact asymptotics for the $L^2$-error of the conditional expectations $\E[\mathcal{O}_T(y)|\mathcal{G}_n]$ and $\E[L_T(y)|\mathcal{G}_n]$ as $n\rightarrow \infty$ with explicit constants. Note that local times are square integrable and therefore  $\E[L_T(y)|\mathcal{G}_n]$ is well-defined (see \cite{marcusrosen2008}). 

\begin{theorem}
\label{Theo2} If $1< \alpha \leq 2$ and $y\in\R$, then
\begin{align*}
 & \lim_{n\rightarrow\infty}\Delta_{n}^{-1-1/\alpha}\norm{\E[\c O_T(y)|\c G_n]-\c O_T(y)}_{L^{2}(\P)}^{2} =2\E[L_T(y)]\int_{0}^{\infty}\E\left[\operatorname{Var}\left(\l{\c O_1(x)}X_{1}\right)\right]dx.
\end{align*}
If $0<\alpha < 1$, then
\[
 \lim_{n\rightarrow\infty}\Delta_{n}^{-2} (\log n)^{-1}\norm{\E[\c O_T(0)|\c G_n]-\c O_T(0)}_{L^{2}(\P)}^{2} =\frac{\Gamma(\alpha)\sin(\frac{\pi \alpha}{2} )}{12 \pi} \E[|X_1|^{-\alpha} ].
\]
If $\alpha=1$, then
\[
 \lim_{n\rightarrow\infty}\Delta_{n}^{-2} (\log n)^{-2}\norm{\E[\c O_T(0)|\c G_n]-\c O_T(0)}_{L^{2}(\P)}^{2} =\frac{1}{12 \pi^2}.
\]
\end{theorem}

In view of Theorems \ref{theorem:OccTime} and \ref{Theo1} we conclude that the Riemann estimator $\hat{\mathcal{O}}_{T,n}(y)$ is rate optimal for all $1<\alpha\leq 2$ and all $y\in\R$, while for $0<\alpha\leq 1$, $\hat{\mathcal{O}}_{T,n}(0)$ is even asymptotically efficient and achieves the minimal possible error. In particular, the Riemann estimator automatically recovers the different regimes for $\alpha$. Efficiency does not hold true for $1<\alpha\leq 2$, in particular not for Brownian motion, as the next proposition shows.
 
\begin{proposition}\label{prop:prop5}
 For all $1<\alpha \leq 2$ we have
\[
 	\frac{\lim_{n\rightarrow\infty}\Delta_{n}^{-1-1/\alpha}\norm{\hat{\c O}_{T,n}(0)-\c O_T(0)}_{L^{2}(\P)}^{2}}{ 
 	 \lim_{n\rightarrow\infty}\Delta_{n}^{-1-1/\alpha}\norm{\E[\c O_T(0)|\c G_n]-\c O_T(0)}_{L^{2}(\P)}^{2}} =: C_{\alpha} > 1.
\]
\end{proposition}
\begin{figure}
\centering
	\includegraphics[width=0.6\textwidth]{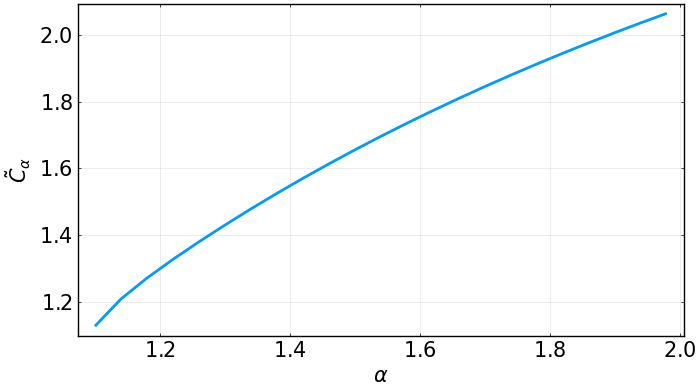}
	\caption{Numerical illustration of the function $\alpha\mapsto \tilde{C}_{\alpha}$.}
\label{fig:1}
\end{figure}
We therefore conclude that the Riemann estimator $\hat{\c O}_{T,n}(0)$ always has a strictly larger estimation error than $\E[\c O_T(0)|\c G_n]$ for all $1<\alpha\leq 2$. To get some idea of how large $C_{\alpha}$ is, let us use  \eqref{eq:prop5_1} from the proof of Proposition \ref{prop:prop5} to lower bound $C_{\alpha}$ by
    \begin{align*}
        C_{\alpha} > \tilde{C}_{\alpha} := \frac{2(2\alpha+1)(\alpha-1)}{\alpha^3} \int_0^\infty \frac{\psi(x)}{x^{2-1/\alpha}} dx.
    \end{align*}
Now $\tilde{C}_{\alpha}$ can be easily evaluated numerically, and Figure \ref{fig:1} shows that $\alpha\mapsto \tilde{C}_{\alpha}$ increases with $\alpha$. In particular, $\tilde{C}_{2}\approx 2.08$. 

We conclude by an exact convergence result for local times at any $y\in\R$, which proves together with Theorem \ref{theorem:Theo3} the rate optimality of the Riemann estimator $\hat{L}_{T,n}(y)$. We conjecture that there is an analogous statement to Proposition \ref{prop:prop5} for local times, but a proof seems difficult.

\begin{theorem}
\label{theorem:Theo4}Let $1<\alpha \leq 2$, $y\in\R$ and set $C(\alpha)=-(\alpha-1)\Gamma(\alpha)\cos(\frac{\pi \alpha}{2})$. Then 
\begin{align*}
  \lim_{n\rightarrow\infty}\Delta_{n}^{1/\alpha-1}\norm{\E[\l{L_{T}(y)}\c G_{n}]-L_{T}(y)}_{L^{2}(\P)}^{2}
  = \frac{\E[L_T(y)]}{C(\alpha)^{2}}\int_{\R}\E[\operatorname{Var}(\l{L_{1}(x)}X_{1})]dx.
\end{align*}
\end{theorem}

\begin{remark}
  For $\alpha=2$ Theorems \ref{Theo2} and \ref{theorem:Theo4} can be obtained from \cite[Theorem 3]{Ivanovs2020}. 
\end{remark}

\begin{remark}Let us discuss some further properties of $\hat{\c O}_{T,n}(A)$ and $\E[\c O_T(A)|\c G_n]$.
$\,$
\begin{enumerate}
\item Section 3 of \cite{Altmeyer2019c} shows that the Riemann estimator is asymptotically efficient as $n\rightarrow \infty$ for approximating integral functionals with smooth integrands when $X$ is a Brownian motion, implying that the integral approximations in \eqref{eq:occLocTime} can not be improved asymptotically by considering higher order quadrature rules such as the trapezoidal rule. 

\item From \eqref{eq:occFormula} in the proof of Lemma \ref{lem:condExp}, we can see that
\begin{equation}
	\E[\c O_T(A)|\c G_n] = \sum_{k=1}^n g_{\alpha}(X_{t_{k-1}}, X_{t_k}-X_{t_{k-1}}),\quad A\subset\R,\label{eq:condExp}
\end{equation}
for a function $g_{\alpha}$ depending explicitly on $\alpha$ and the density $f_{\alpha}$. This suggests that an asymptotically efficient estimator for $1<\alpha\leq 2$ needs to use also the independent increments  $X_{t_k}-X_{t_{k-1}}$ besides the time points $X_{t_{k-1}}$, but this seems not to be necessary for $0<\alpha\leq 1.$ 

\item The identity \eqref{eq:condExp} suggests to use the conditional expectation as estimator. While it is possible to pre-estimate $\alpha$ from the observations, cf. \cite{belomestnySpectralEstimationFractional2010}, $f_{\alpha}$ is generally not known in closed form and needs to be numerically approximated. In view of the good performance of the Riemann estimators across different models, cf. \cite{altmeyerApproximationOccupationTime2021}, one should refrain from using \eqref{eq:condExp} as estimator if $X$ is not precisely a symmetric $\alpha$-stable process.
\end{enumerate}

\end{remark}

\section{Proofs}\label{sec:proofs}

All along the proofs, $Z$, $\tilde{Z}$ are generic standard $\alpha$-stable random variables, independent of each other. For two non-negative functions $f$ and $g$, we write $f \lesssim g$ if $\sup \frac{f}{g} < +\infty$.

\subsection{Proofs of results for the Riemann estimators}

\subsubsection{Proof of Theorem \ref{Theo1}}

Let us decompose $V_{T,n}:=\norm{\c O_T(0) -  \hat{\c O}_{T,n}(0)}_{L^{2}(\P)}^{2}$
as 
\begin{align}
 & \norm{\c O_T(0)}_{L^{2}(\P)}^{2}+\norm{\hat{\c O}_{T,n}(0)}_{L^{2}(\P)}^{2}-2\E[\c O_T(0) \hat{\c O}_{T,n}(0)].\label{eq:Theo1_1}
\end{align}
We first compute these three expression explicitly. For this, observe the following properties of marginal and joint densities of the standard $\alpha$-stable distribution.

\begin{lemma}
\label{Lem1} Define 
\[
\phi(x,y)=\P(Z\geq0,\tilde{Z}\geq0,x{}^{1/\alpha}Z\leq y^{1/\alpha}\tilde{Z}),\,\,\,\,x,y\geq0,
\]
Then we have for $0<r<t$ that $\P(X_{r}\geq0,X_{t}\geq0)=\frac{1}{4}+\phi(t-r,r)$.
\end{lemma}

\begin{proof}
Observe for $0<r<t$ by symmetry and independence of increments that $\P(X_{r}\geq0,X_{t}-X_{r}\geq0)=\frac{1}{4}$. Therefore
\begin{equation*}
\P(X_{r}\geq0,X_{t}\geq0) =  \frac{1}{4}+\P(X_{r}\geq0,0\geq X_{t}-X_{r}\geq-X_{r}).
\end{equation*}
By the $1/\alpha$-self-similarity we can write $X_{t}-X_{r}\overset{d}{=}(t-r)^{1/\alpha}Z$, $X_{r}\overset{d}{=}r^{1/\alpha}\tilde{Z}$. Symmetry yields $Z\overset{d}{=}-Z$ such that
\begin{align*}
\P(X_{r}\geq0,X_{t}\geq0) & =  \frac{1}{4}+\P(\tilde{Z}\geq0,0\leq(t-r)^{1/\alpha}Z\leq r^{1/\alpha}\tilde{Z}) =  \frac{1}{4}+\phi(t-r,r).
\end{align*}
\end{proof}

\begin{lemma}
\label{lem:lem2}The function $\phi$ from Lemma \ref{Lem1} enjoys the following properties for $x,y,b>0$, $0<a\leq T$:
\begin{enumerate}
\item $\phi(x,y)+\phi(y,x)=\frac{1}{4},$ 
\item $\phi(ax,ay)=\phi(x,y)$, 
\item $\int_{0}^{a}\phi(a-x,x)dx=\frac{a}{8}$, 
\item $\int_{a}^{T}\phi(x-a,a)dx=\frac{1}{4}\E[\min(T,aD^{-1})]-\frac{a}{4}$
with $D=(1+|\tilde{Z}/Z|^{\alpha})^{-1}$. 
\end{enumerate}
\end{lemma}

\begin{proof}
(i) follows from $(Z,\tilde{Z})\overset{d}{=}(\tilde{Z},Z)$, (ii) is
clear. For (iii), substitution shows $\int_{0}^{a}\phi(x,a-x)dx=\int_{0}^{a}\phi(a-x,x)dx$
and the claim follows from (i) and 
\begin{align*}
\int_{0}^{a}\phi(x,a-x)dx & =\frac{1}{2}(\int_{0}^{a}\phi(x,a-x)dx+\int_{0}^{a}\phi(a-x,x)dx).
\end{align*}
For $(iv)$, on the other hand, we have
\begin{align*}
\int_{a}^{T}\phi(x-a,a)dx & =\E[\I_{\{ Z\geq 0\} }\I_{\{\tilde{Z}\geq 0 \} }\int_{a}^{T}\I_{\{ (x-a)|Z|^{\alpha}\leq a|\tilde{Z}|^{\alpha}\} }dx]\\
 & =\E[\I_{\{Z\geq 0\} }\I_{\{\tilde{Z}\geq 0\} }\int_{a}^{T}\I_{\{ x\leq a\frac{|Z|^{\alpha}+|\tilde{Z}|^{\alpha}}{|Z|^{\alpha}}\} }dx]\\
 & =\frac{1}{4} \E[\min(T,a(1+|\tilde{Z}/Z|^{\alpha}))]-\frac{a}{4}.\qedhere
\end{align*}
\end{proof}
\begin{lemma}\label{lem:T1} 
We have $\norm{\c O_T(0)}_{L^{2}(\P)}^{2} =\frac{3}{8}T^{2}$.
\end{lemma}
\begin{proof} Use Lemma \ref{Lem1}
to obtain 
\begin{align*}
\norm{\c O_T(0)}_{L^{2}(\P)}^{2} & =  2\int_{0}^{T}\int_{r}^{T}\P(X_{r}\geq0,X_{t}\geq0)dtdr\\
 & =  2\int_{0}^{T}\int_{r}^{T}\frac{1}{4}dtdr+2\int_{0}^{T}\int_{0}^{T-r}\phi(x,r)dxdr.
\end{align*}
Consequently, by changing the variables $x,r$ and using Lemma \ref{lem:lem2}(i)
\begin{align*}
\norm{\c O_T(0)}_{L^{2}(\P)}^{2} & =  \frac{T^{2}}{4}+\int_{0}^{T}\int_{0}^{T-r}\phi(x,r)dxdr+\int_{0}^{T}\int_{0}^{T-x}\phi(x,r)drdx\\
 & =  \frac{T^{2}}{4}+\int_{0}^{T}\int_{0}^{T-r}(\phi(x,r)+\phi(r,x))dxdr\\
 & =  \frac{T^{2}}{4}+\int_{0}^{T}\int_{0}^{T-r}\frac{1}{4}dxdr=\frac{3}{8}T^{2}.\qedhere
\end{align*}
\end{proof}

\begin{lemma}\label{lem:T2}
We have $\norm{\hat{\c O}_{T,n}(0)}_{L^{2}(\P)}^{2} =\frac{3}{8}T^{2}+\frac{3}{8}T\Delta_{n}+\frac{1}{4}\Delta_{n}^{2}$.
\end{lemma}
\begin{proof}
Since $X_{0}=0$, expanding
the square shows that $\norm{\hat{\c O}_{T,n}(0)}_{L^{2}(\P)}^{2}$
equals 
\begin{align*}
 & \E[(\Delta_{n}+\Delta_{n}\sum_{k=1}^{n-1}\I_{[0,\infty)}(X_{t_k}))^{2}]=\Delta_{n}^{2}+3\Delta_{n}^{2}\sum_{k=1}^{n-1}\P(X_{t_k}\geq0)\\
 & \qquad+2\Delta_{n}^{2}\sum_{k=1}^{n-2}\sum_{j=k+1}^{n-1}\P(X_{t_k}\geq0,X_{j\Delta_{n}}\geq0).
\end{align*}
As the distribution of $X_{t}$ is symmetric for $t>0$, the first
two terms are just $\Delta_{n}^{2}+\frac{3}{2}\Delta_{n}^{2}(n-1)$.
Together with Lemma \ref{Lem1}, the last display is thus equal to
\begin{align*}
 & \Delta_{n}^{2}+\frac{3}{2}\Delta_{n}^{2}(n-1)+2\Delta_{n}^{2}\sum_{k=1}^{n-2}\sum_{j=k+1}^{n-1}\frac{1}{4}+2\Delta_{n}^{2}\sum_{k=1}^{n-2}\sum_{j=k+1}^{n-1}\phi((j-k)\Delta_{n},t_k).
\end{align*}
By Lemma \ref{lem:lem2}(ii) and an index change in the last sum
we have 
\begin{align*}
 & 2\sum_{k=1}^{n-2}\sum_{j=k+1}^{n-1}\phi((j-k)\Delta_{n},t_k)=2\sum_{k=1}^{n-2}\sum_{j=1}^{n-k-1}\phi(j,k)=\sum_{k=1}^{n-2}\sum_{j=1}^{n-k-1}(\phi(j,k)+\phi(k,j)).
\end{align*}
Lemma \ref{lem:lem2}(i) thus implies
\begin{align*}
\norm{\hat{\c O}_{T,n}(0)}_{L^{2}(\P)}^{2} & =\Delta_{n}^{2}\frac{3n-1}{2}+2\Delta_{n}^{2}\sum_{k=1}^{n-2}\sum_{j=k+1}^{n-1}\frac{1}{4}+\Delta_{n}^{2}\sum_{k=1}^{n-2}\sum_{j=1}^{n-k-1}\frac{1}{4}\\
 & =\frac{3}{8}T^{2}+\frac{3}{8}T\Delta_{n}+\frac{1}{4}\Delta_{n}^{2}.\qedhere
\end{align*}
\end{proof}

\begin{lemma}\label{lem:T3}
We have $2\E[\c O_T(0) \hat{\c O}_{T,n}(0)] =\frac{3}{4}T^{2}+\frac{3}{8}T\Delta_{n}-\frac{1}{8}\Delta_n^{2} \E[\frac{\psi(nD)}{D(1-D)}]$ with $D=(1+|\tilde{Z}/Z|^{\alpha})^{-1}$.
\end{lemma}
\begin{proof}
Let $A=2\E[\c O_T(0) \hat{\c O}_{T,n}(0)]$.
Since $X_{0}=0$, Lemma \ref{Lem1} and splitting the integral into two parts show
\begin{align*}
A & =2\Delta_{n}\int_{0}^{T}\P(X_{r}\geq0)dr+2\Delta_{n}\sum_{k=1}^{n-1}\int_{0}^{T}\P(X_{t_k}\geq0,X_{r}\geq0)dr\\
 & =T\Delta_{n}+2\Delta_{n}\sum_{k=1}^{n-1}\int_{0}^{T}\frac{1}{4}dr+2\Delta_{n}\sum_{k=1}^{n-1}\int_{0}^{t_k}\phi(t_k-r,r)dr\\
 & \qquad+2\Delta_{n}\sum_{k=1}^{n-1}\int_{t_k}^{T}\phi(r-t_k,t_k)dr.
\end{align*}
We thus get from Lemma \ref{lem:lem2}(iii), (iv) that 
\begin{equation*}
 A=T\Delta_{n}+\frac{T\Delta_{n}}{2}(n-1)+\Delta_{n}\sum_{k=1}^{n-1}\frac{t_k}{4}+\frac{\Delta_{n}}{2}(\sum_{k=1}^{n-1}\E[\min(T,\frac{k}{D}\Delta_{n})]-\sum_{k=1}^{n-1}t_k).
\end{equation*}
Since $\min(T,\frac{k}{D}\Delta_{n})=T\min(1,\frac{k}{nD})$,
this means that 
\begin{align*}
\sum_{k=1}^{n-1}\min(T,\frac{k}{D}\Delta_{n}) & =T\sum_{k=1}^{n-1}\min(1,\frac{k}{nD})=T\sum_{k=1}^{\lfloor nD\rfloor}\frac{k}{nD}+T(n-1-\lfloor nD\rfloor)\\
 & =T\frac{\lfloor nD\rfloor(\lfloor nD\rfloor+1)}{2nD}-T\lfloor nD\rfloor+T(n-1)\\
 & =\frac{T}{2}(1-\frac{\psi(nD)}{nD}-nD)+T(n-1).
\end{align*}
Noting that $D\overset{d}{=}1-D$, we have $\E[D]=1/2$. The last line therefore has expectation
\[
-\frac{T}{2}-\frac{\Delta_n}{2}\E\left[\frac{\psi(nD)}{D}\right]+\frac{3Tn}{4}=-\frac{T}{2}-\frac{\Delta_n}{4}\E\left[\frac{\psi(nD)}{D(1-D)}\right]+\frac{3Tn}{4},
\]
where we used for the last equality the relation $\psi(nx)=\psi(n(1-x))$, which follows from $\lfloor x\rfloor + \lfloor n-x\rfloor = n-1$. From this obtain the result.
\end{proof}

\begin{proof}[Proof of Theorem \ref{Theo1}]
The decomposition \eqref{eq:Theo1_1} and Lemmas \ref{lem:T1}, \ref{lem:T2}, \ref{lem:T3} show 
\begin{align*}
V_{T,n} & =\frac{\Delta_{n}^{2}}{4}+\frac{\Delta_n^{2}}{8} \E\left[\frac{\psi(nD)}{D(1-D)}\right].
\end{align*}
Denote the Lebesgue density of $D=(1+|\tilde{Z}/Z|^{\alpha})^{-1}$ by $f_{D}$. By the symmetry $D\overset{d}=1-D$, we have $f_{D}(x)=f_{D}(1-x)$. A change of variables combined with the equality $\psi(nx) = \psi(n(1-x))$ thus allows for rewriting the last display as
\[
V_{T,n} = \frac{\Delta_n^2}{4} + \frac{\Delta_n^{2}}{4} \int_0^{n/2} \frac{\psi(x)}{x(1-x/n)}f_{D}\left(\frac{x}{n}\right) dx.\label{eq:V}
\]
The result follows from studying the $dx$-integral integral as $n\rightarrow\infty$ for different $\alpha$. For $1<\alpha\leq 2$, we find from Lemma \ref{lem:lem4}(i) and $0<x<n/2$ that
\[
n^{-1+1/\alpha}\frac{\psi(x)}{x(1-x/n)}f_{D}\left(\frac{x}{n}\right)
	\lesssim \frac{\psi(x)}{x^{2-1/\alpha}},
\]
which is integrable for $x>0$. Part (i) of the theorem follows then immediately from the dominated convergence theorem and the convergence in Lemma \ref{lem:lem4}(i). For part (ii) and $0<\alpha<1$, note that $f_{D}$ is bounded and the limit $f_D(0+):=\lim_{x\rightarrow 0}f_{D}(x)$ exists according to Lemma \ref{lem:lem4}(ii). By Lemma \ref{lem:lem3}(i), we conclude that
\begin{align*}
& \frac{1}{\log n}\int_0^{n/2}\frac{\psi(x)}{x(1-x/n)}f_{D}\left(\frac{x}{n}\right)dx
	= \frac{f_{D}(0+)}{\log n}\int_0^{n/2}\frac{\psi(x)}{x(1-x/n)}dx \\
	&  \quad\quad+\frac{1}{\log n} \int_0^{n/2}\frac{\psi(x)}{x(1-x/n)}(f_{D}\left(\frac{x}{n}\right)-f_{D}(0+))dx = \frac{f_{D}(0+)}{6}+o(1).
\end{align*}
Finally, for part (iii) and $\alpha=1$, we find from Lemmas \ref{lem:lem4}(iii) and \ref{lem:lem3}(ii) that
\begin{align*}
\frac{1}{(\log n)^2}\int_0^{n/2}\frac{\psi(x)}{x(1-x/n)}f_{D}\left(\frac{x}{n}\right)dx
	& = \frac{1}{(\log n)^2}\frac{4}{\pi^2}\int_0^{n/2}\frac{\psi(x)\log((x/n)^{-1}-1)}{x(1-x/n)(1-2x/n)}dx
\end{align*}
converges to $\frac{1}{3\pi^2}$, thereby implying the claimed convergence. This finishes the proof.
\end{proof}

\subsubsection{Proof of Proposition \ref{prop:prop5}}

According to Theorems \ref{Theo2} and \ref{Theo1} for $1<\alpha\leq 2$ it suffices to show
\[ 
	\frac{2T^{1-1/\alpha}\frac{\Gamma(1/\alpha)}{\pi \alpha^2}\E[|Z|]\int_{0}^{\infty}\frac{\psi(x)}{x^{2-1/\alpha}}dx}{2\E[L_T(0)]\int_{0}^{\infty}\E\left[\operatorname{Var}\left(\l{\c O_1(x)}Z\right)\right]dx} >1.
\]
From \eqref{expressionautosim} and the occupation time formula \eqref{occfor} we infer $\E[L_T(0)]=\int_0^T f_{\alpha,t}(0)dt = \frac{\Gamma(1/\alpha)}{\pi(\alpha-1)}T^{1-1/\alpha}$. The result follows from
\begin{align}
& \int_{0}^{\infty}\E\left[\operatorname{Var}\left(\l{\c O_1(x)}Z\right)\right]dx < \frac{\alpha}{2(2\alpha+1)}\E[|Z|], \label{eq:prop5_1} \\ 
& \int_{0}^{\infty}\frac{\psi(x)}{x^{2-1/\alpha}}dx \geq \frac{\alpha^3}{2(\alpha-1)(2\alpha+1)}. \label{eq:prop5_2}
\end{align}
Recall \eqref{expressionrho} in Lemma \ref{lem:min} with $\rho_{r,s}$ being the density of $X_r\wedge X_s$ such that
\begin{align*}
& \int_0^\infty \E\left[\text{Var}\left(\l{\c O_1(x)}Z\right)\right]dx 
	 < \int_0^{\infty} \E\left[\left(\int_{0}^{1}\1_{X_{r}\geq x}dr\right)^{2}\right]dx\\
 	& \quad = \int_0^{\infty} \int_0^1 \int_r^1 \P(X_r \wedge X_s \geq x )ds dr dx
	 = \int_0^1 \int_r^1 \int_0^{\infty} u \rho_{r,s}(u)du ds dr\\
	& \quad = \frac{\alpha^2}{2(\alpha+1)(2\alpha+1)} \int_0^{\infty} u f_{\alpha}(u)du \\
	&\quad\quad+ \int_0^1 	\int_r^1 \int_0^{\infty} \int_0^{\infty} u f_{\alpha,r}(u+v)f_{\alpha,s-r}(v) dv du ds dr\\
 & \quad = \frac{\alpha^2}{4(\alpha+1)(2\alpha+1)} \E[|Z| ]+ \int_0^1 \int_r^1 \E[X_s\1_{X_s \geq 0} \1_{X_r \leq 0} ] ds dr.
\end{align*}
By the symmetry $X\overset{d}{=}-X$ we deduce for $s>r$
\begin{align*}
0 & =\E[X_s (\1_{X_s \geq 0} \1_{X_r \leq 0}+ \1_{X_s \leq 0} \1_{X_r \geq 0}) ] \\
   & = \E[X_s (2 \1_{X_s \geq 0} \1_{X_r \leq 0} +1 - \1_{X_s\geq 0} - 1_{X_r \leq 0}) ]
\end{align*}
and therefore
\begin{align*}
 2 \E[X_s\1_{X_s \geq 0} \1_{X_r \leq 0} ] &= \E[ X_s \1_{X_s \geq 0}]- \E[ X_s\1_{X_r \leq 0}]\\
 &= \E[ X_s \1_{X_s \geq 0}] - \E[ (X_s -X_r) \1_{X_r \leq 0}] - \E[ X_r \1_{X_r \leq 0}]\\
 & = \frac{s^{1/\alpha}-r^{1/\alpha}}{2}\E[ |Z|].
\end{align*}
From this obtain \eqref{eq:prop5_1}. To conclude let us compute $\int_{0}^{\infty}\frac{\psi(x)}{x^{2-1/\alpha}}dx = \int_0^1 \frac{x-x^2}{x^{2-1/\alpha}}dx + \sum_{k=1}^{\infty} \int_0^1 \frac{x-x^2}{(x+k)^{2-1/\alpha}}dx$. Integration by parts entails
\begin{align*}
 \int_{0}^{\infty}\frac{\psi(x)}{x^{2-1/\alpha}}dx &=\frac{\alpha^2}{\alpha+1} + \frac{\alpha^2}{\alpha-1} \sum_{k=1}^{\infty} \frac{2\alpha}{\alpha+1} ((1+k)^{1+1/\alpha}-k^{1+1/\alpha} )-k^{1/\alpha} -(1+k)^{1/\alpha}\\
 &=\frac{\alpha^2}{\alpha+1} + \frac{\alpha^2}{\alpha-1} \sum_{k=1}^{\infty} \frac{2\alpha}{\alpha+1} k^{1+1/\alpha}[(1+\frac{1}{k})^{1+1/\alpha}-1 ]-k^{1/\alpha} - k^{1/\alpha}(1+\frac{1}{k})^{1/\alpha}.
\end{align*}
We then use the inequalities
\[(1+x)^{\alpha'}-1 \geq \alpha' x+ \frac{\alpha'(\alpha'-1)}{2} x^2 - \frac{\alpha'(\alpha'-1)(2-\alpha')}{6}x^3\]
for $\alpha' \in (1,2)$ and $x \in (0,1),$ as well as 
\[(1+x)^{\alpha''}\leq 1+ \alpha'' x- \frac{\alpha''(1-\alpha'')}{2} x^2 + \frac{\alpha''(1-\alpha'')(2-\alpha'')}{6}x^3\]
for $\alpha'' \in(0,1)$ and $x\in (0,1)$ to obtain
\begin{align*}
 \int_{0}^{\infty}\frac{\psi(x)}{x^{2-1/\alpha}}dx &\geq \frac{\alpha^2}{\alpha+1} + \frac{\alpha^2}{\alpha-1} \sum_{k=1}^{\infty}  [\frac{2\alpha}{\alpha+1} k^{1+1/\alpha}[ \frac{\alpha+1}{\alpha k} + \frac{\alpha+1}{2\alpha^2 k^2}- \frac{(\alpha+1)(\alpha-1)}{6\alpha^3 k^3}]  \\
 &- k^{1/\alpha} - k^{1/\alpha}(1+ \frac{1}{\alpha k} - \frac{\alpha-1}{2\alpha^2 k^2} + \frac{(\alpha-1)(2\alpha-1)}{6\alpha^3 k^3} )]\\
 &= \frac{\alpha^2}{\alpha+1} + \frac{1}{6} \sum_{k=1}^{\infty} k^{1/\alpha -2} - \frac{(2\alpha-1)}{6\alpha} \sum_{k=1}^{\infty} k^{1/\alpha -3}\\
 & \geq \frac{\alpha^2}{\alpha+1} + \frac{1}{6} \int_{1}^{\infty} x^{1/\alpha -2} dx - \frac{(2\alpha-1)}{6\alpha} \int_{0}^{\infty} x^{1/\alpha -3}dx\\
 &= \frac{6\alpha^3 -6\alpha^2 +\alpha+1}{6(\alpha-1)(\alpha+1)} 
 	> \frac{\alpha^3}{2(\alpha-1)(2\alpha+1)},
\end{align*}
which holds for $\alpha>1$. This yields \eqref{eq:prop5_2} and finishes the proof.

\subsubsection{Proof of Theorem \ref{theorem:OccTime}}

Let $g=\I_{[a,b]}$, $0\leq t,t'\leq T$ and define
 \begin{align}
  E_{t,t'} &:= \E[(g(X_t)-g(X_{\Delta_n \lfloor t/\Delta_n \rfloor} ))(g(X_{t'})-g(X_{\Delta_n \lfloor t'/\Delta_n \rfloor }))] \nonumber.
 \end{align}
The main idea of the proof is to get a tight control on $E_{t,t'}$ using Fourier calculus, see Lemma \ref{lem:marginalDifferences} below. Note that $\Delta_n \lfloor t/\Delta_n \rfloor=t_{k-1}$, if $t_{k-1}\leq t < t_k$. On the time interval $[0,\Delta_n]$, the estimation error $\norm{\c O_{\Delta_n}([a,b])-\hat{\c O}_{\Delta_n,n}([a,b])}_{L^2(\P)}$ is bounded by $\Delta_n$, which allows us to write
\begin{align*}
& \norm{\c O_{T}([a,b])-\hat{\c O}_{T,n}([a,b])}^2_{L^2(\P)}  \leq 2\Delta_n^2  +4( A_1 + A_2),\\
& \qquad \text{with}\quad  A_1  = \sum_{k=2}^n\sum_{k'=k+1}^n\int_{t_{k-1}}^{t_k}\int_{t_{k'-1}}^{t_{k'}}E_{t,t'}dt'dt,
\quad A_2  = \sum_{k=2}^n\int_{t_{k-1}}^{t_k}\int_{t}^{t_k}E_{t,t'}dt'dt.
\end{align*}
Lemma \ref{lem:marginalDifferences}(i,ii) gives
\begin{align*}
	|A_1| 
		& \lesssim (1\vee T^{\epsilon/\alpha})\Delta_n^2 \sum_{k=2}^n (t_{k-1}^{-(1+\epsilon)/\alpha}\Delta_n^{1/\alpha}  + t_{k-1}^{-1-\epsilon/\alpha} \Delta_n \log n) \\
		 & \lesssim (1\vee T^{\epsilon/\alpha})(T^{1-1/\alpha-\epsilon/\alpha} \Delta_n^{1+1/\alpha} + \Delta_n^{2-\epsilon/\alpha}\log n) \lesssim (1\vee T^{\epsilon/\alpha})T^{1-1/\alpha-\epsilon/\alpha} \Delta_n^{1+1/\alpha},\\
  |A_2| & \lesssim  \Delta_n^2 \sum_{k=2}^n
		(\Delta_n^{1/\alpha} (t_{k-1}^{-1/\alpha}+ t_{k-1}^{-(1+\epsilon)/\alpha}) + \Delta_n^{1-\epsilon/\alpha}t_{k-1}^{-1-\epsilon/\alpha} + \Delta_n^{1-2\epsilon/\alpha}t_{k-1}^{-1})\\
		& \lesssim (1\vee T^{2\epsilon/\alpha})T^{1-1/\alpha-2\epsilon/\alpha} \Delta_n^{1+1/\alpha},
\end{align*}
using that $n^{1/\alpha+\epsilon/\alpha-1}\log n \leq 1$ for $\alpha > 1$ and $\epsilon$ small enough. The result follows from modifying $\epsilon$. 

\begin{lemma}\label{lem:marginalDifferences}
The following holds:
\begin{enumerate}
\item If $t_{k-1}\leq t < t_k$ and $t_{k'-1}\leq t'< t_{k'}$, $k\neq k'$, then for sufficiently small $\epsilon>0$
\[
	|E_{t,t'}| \lesssim (1\vee T^{\epsilon/\alpha}) (t_{k-1}^{-(1+\epsilon)/\alpha}\Delta_n^{1/\alpha}  + t_{k-1}^{-1-\epsilon/\alpha} \Delta_n \log n).
\]
\item If $t_{k-1}<t<t'<t_k$, then for sufficiently small $\epsilon>0$
\[
|E_{t,t'}| 
		 \lesssim \Delta_n^{1/\alpha} (t_{k-1}^{-1/\alpha}+ t_{k-1}^{-(1+\epsilon)/\alpha}) + \Delta_n^{1-\epsilon/\alpha}t_{k-1}^{-1-\epsilon/\alpha} + \Delta_n^{1-2\epsilon/\alpha}t_{k-1}^{-1}.
\]
\end{enumerate}
\end{lemma}
\begin{proof}
Assume first that $-\infty < a < b < \infty$. The Plancherel theorem (on $\R^2$) shows for $0<t<t'\leq T$ that 
\begin{equation}
	\E[g(X_t)g(X_{t'})] = (2\pi)^{-2}\int_{\R^2} \c F g(u)\c F g(v) \phi(-u,t;-v,t')d(u,v), \label{eq:E_g}
\end{equation}
where $\phi(u,t;v,t')=\E[e^{iuX_t+ivX_t'}]$ is the characteristic function of $(X_t,X_{t'})$ and 
$\c Fg(u)=\int_{\R}g(x)e^{iux}dx=(iu)^{-1}(e^{iub}-e^{iua})$ is the Fourier transform of $g$. By independence of increments, we have $\phi(u,t;v,t')=e^{-|u+v|^{\alpha}t-|v|^{\alpha}(t'-t)}$ for $t'>t$. From \eqref{eq:E_g}
\begin{align}
	& |E_{t,t'}| \lesssim \int_{\R^2}\c (1\wedge |u|^{-1}) (1\wedge |v|^{-1}) |E_{t,t'}^{-u,-v}|d(u,v),\label{eq:E_upperbound}
\end{align}
where $E_{t,t'}^{u,v} 
	 = \phi(u,t;v,t')-\phi(u,t;v,t_{k'-1}) -\phi(u,t_{k-1};v,t')+\phi(u,t_{k-1};v,t_{k'-1})$. 
By an approximation argument, this upper bound is also true for any $-\infty \leq a \leq b\leq \infty$. Consider now first $t_{k-1}\leq t < t_k$ and $t_{k'-1}\leq t'< t_{k'}$, $k\neq k'$. Then 
\begin{align}
	E_{t,t'}^{u,v}  & = \int_{t_{k-1}}^{t}\int_{t_{k'-1}}^{t'}\partial^2_{r,r'}\phi(u,r;v,r')dr'dr\nonumber \\
	& =  \int_{t_{k-1}}^t\int_{t_{k'-1}}^{t'} (|v|^{\alpha}|u+v|^{\alpha}-|v|^{2\alpha})\phi(u,r;v,r')dr'dr. \nonumber
\end{align}
Using $|v|\leq |u|+|u+v|$, as well as distinguishing the cases $|u+v|\leq |u|$ and $|u+v|>|u|$, we have on the one hand for $0<\epsilon < 1$  
\begin{align}
	 (1\wedge |u|^{-1})(1\wedge|v|^{-1})
	 & \leq (1\wedge |u|^{-1})|v|^{-2}(2|u| + 2|u+v|^{1+\epsilon}|u|^{-\epsilon})\nonumber \\
	& \leq  2|v|^{-2}(1 + (|u|^{-\epsilon}\wedge |u|^{-1-\epsilon})|u+v|^{1+\epsilon}),\label{eq:v_bound}
	\end{align}
	and on the other hand
	\begin{align}
	(1\wedge |u|^{-1})(1\wedge|v|^{-1})|u+v|^{\alpha} & \leq (1\wedge|v|^{-1}) (|u+v|^{\alpha-1} + (|u|^{-\epsilon}\wedge|u|^{-1-\epsilon})|u+v|^{\alpha+\epsilon}) \nonumber\\
	& \leq |v|^{-1}(|u+v|^{\alpha-1} + (|u|^{-\epsilon}\wedge|u|^{-1-\epsilon})|u+v|^{\alpha+\epsilon}).\label{eq:uplusv_v_bound}
\end{align}
Note that $|u+v|^{\beta}e^{-|u+v|^{\alpha}r/2} \lesssim r^{-\beta /\alpha}$ for $\beta\geq 0$, $\int_{\R}|v|^{\gamma}e^{-|v|^{\alpha}(r'-r)}dv\lesssim (r'-r)^{-(\gamma+1)/\alpha}$ for $\gamma > -1$, as well as $\int_{\R}e^{-|u+v|^{\alpha}r}du\lesssim r^{-1/\alpha}$, $\int_{\R}(|u|^{-\epsilon}\wedge|u|^{-1-\epsilon})du\lesssim 1$. Multiplying  \eqref{eq:v_bound} by $|v|^{2\alpha}$ and \eqref{eq:uplusv_v_bound} by $|v|^{\alpha}$, yields then in \eqref{eq:E_upperbound}, as long as $1+\epsilon <\alpha$ and using the upper bound $r^{\epsilon/\alpha}\leq T^{\epsilon/\alpha}$,
\begin{align*}
	|E_{t,t'}| 
		& \lesssim  \int_{t_{k-1}}^{t_k}\int_{t_{k'-1}}^{t_{k'}} ((r'-r)^{-2+1/\alpha}(r^{-1/\alpha}+r^{-(1+\epsilon)/\alpha}) + (r'-r)^{-1}(r^{-1}+r^{-1-\epsilon/\alpha})) dr'dr.\\
		& \lesssim (1\vee T^{\epsilon/\alpha})\int_{t_{k-1}}^{t_k}\int_{t_{k'-1}}^{t_{k'}}	((r'-r)^{-2+1/\alpha}r^{-(1+\epsilon)/\alpha} + (r'-r)^{-1}r^{-1-\epsilon/\alpha}) dr'dr\\
		& \lesssim (1\vee T^{\epsilon/\alpha}) (t_{k-1}^{-(1+\epsilon)/\alpha}\Delta_n^{1/\alpha}  + t_{k-1}^{-1-\epsilon/\alpha} \Delta_n \log n).
\end{align*}
This proves (i). Let now $t_{k-1}<t<t'<t_k$. To compute $\E[g(X_{t_{k-1}})^2]$, we use \eqref{eq:E_g} to obtain $\E[g(X_{t_{k-1}-\epsilon})g(X_{t_{k-1}+\epsilon})]$ and let $\epsilon \rightarrow 0$. Then, \eqref{eq:E_upperbound} still holds, but this time with
\begin{align*}
	E_{t,t'}^{u,v} & = \int_{t_{k-1}}^{t} (\partial_{r} \phi(u,r;v,t')-\partial_{r} \phi(u,r;v,t_{k-1})) dr \\
					   & = \int_{t_{k-1}}^{t} ((|v|^{\alpha}-|u+v|^{\alpha}) \phi(u,r;v,t')+|u|^{\alpha} \phi(u,r;v,t_{k-1})) dr.
\end{align*}
Similar as above, after multiplying \eqref{eq:v_bound} by $|v|^{\alpha}$, we have
\[
	\int_{\R^2} (1\wedge |u|^{-1})(1\wedge|v|^{-1})|v|^{\alpha} \phi(u,r;v,t')d(u,v)\lesssim  (t'-r)^{-1+1/\alpha}(r^{-1/\alpha}+r^{-(1+\epsilon)/\alpha}).
\]
Moreover, by symmetry in $u,v$ and because $\phi(u,r,v,t_{k-1} )=\phi( v,t_{k-1},u,r)$, the same upper bound follows with respect to $|u|^{\alpha} \phi(u,r;v,t_{k-1})$ when $t'-r,r$ are replaced by $r-t_{k-1},t_{k-1}$. At last, 
using $|u|^{\epsilon}\lesssim |u+v|^{\epsilon}+|v|^{\epsilon}$ and arguing as after \eqref{eq:uplusv_v_bound}, we find that
\begin{align*}
& \int_{\R^2}(1\wedge |u|^{-1})(1\wedge|v|^{-1}) |u+v|^{\alpha}\phi(u,r;v,t')d(u,v)\\
& \quad  \lesssim \int_{\R^2} |u+v|^{\alpha}|v|^{\epsilon}(|u+v|^{\epsilon}+|v|^{\epsilon}) (|u|^{-\epsilon}\wedge |u|^{-1-\epsilon})(|v|^{-\epsilon}\wedge|v|^{-1-\epsilon})\phi(u,r;v,t')d(u,v)\\
& \quad \lesssim r^{-1-\epsilon/\alpha}(t'-r)^{-\epsilon/\alpha}+r^{-1}(t'-r)^{-2\epsilon/\alpha}.
\end{align*}
Combining the last two displays, we find as in (i) from \eqref{eq:E_upperbound} for sufficiently small $\epsilon$ 
\begin{align*}
	|E_{t,t'}| 
		& \lesssim \Delta_n^{1/\alpha} (t_{k-1}^{-1/\alpha}+ t_{k-1}^{-(1+\epsilon)/\alpha}) + \Delta_n^{1-\epsilon/\alpha}t_{k-1}^{-1-\epsilon/\alpha} + \Delta_n^{1-2\epsilon/\alpha}t_{k-1}^{-1}.\qedhere
\end{align*}
\end{proof}


\subsubsection{Proof of Theorem \ref{theorem:Theo3}}

Note that $\hat{L}_{T,n}(y)=(2 h_n)^{-1}\hat{\mathcal{O}}_{T,n}([y-h_n,y+h_n])$. By Theorem \ref{theorem:OccTime} for $1<\alpha\leq 2$ and sufficiently small $\epsilon>0$ we then get
\begin{align}
 \norm{\hat{L}_{T,n}(y)-L_{T}(y)}^2_{L^{2}(\P)} & \lesssim 
	\norm{(2h_n)^{-1}\mathcal{O}_{T}([y-h_n,y+h_n])-L_{T}(y)}^2_{L^{2}(\P)} \nonumber\\ 
	    & + (1\vee T^{\epsilon})h_n^{-2} T^{1-1/\alpha-\epsilon} \Delta_{n}^{1+1/\alpha}.\label{eq:localTimeError}
\end{align}
The occupation time formula \eqref{occfor} yields  $\mathcal{O}_{T}([y-h_n,y+h_n])=\int_{[y-h_n,y+h_n]}L_T(x)dx$ such that because of $\int_{[y-h_n,y+h_n]}dx=2h_n$ 
\begin{align*}
	 & \norm{(2h_n)^{-1}\mathcal{O}_{T}([y-h_n,y+h_n])-L_{T}(y)}_{L^{2}(\P)} 
	 \leq  \frac{1}{2} \int_{-1}^{1}\norm{L_T(y+h_n x)-L_{T}(y)}_{L^{2}(\P)} dx \\
	 & \quad = \frac{T^{1-1/\alpha}}{2}   \int_{-1}^{1}\norm{L_1(T^{-1/\alpha}(y+h_nx))-L_{1}(T^{-1/\alpha}y)}_{L^{2}(\P)} dx,
\end{align*}
where we used the self-similarity of $X$, which also transfers to its local time, cf. Proposition 10.4.8 of \cite{samorodnitskyStochasticProcessesLong2016}. A typical moment bound for local times, see for example \cite[Lemma 5.2]{marcusrosen2008}, therefore implies that the last display is upper bounded up to a constant by 
\begin{align*}
	 T^{1-1/\alpha}  (T^{-1/\alpha}h_n)^{(\alpha-1)/2} = T^{1/2-1/(2\alpha)}h_n^{(\alpha-1)/2}.
\end{align*}
Using this in \eqref{eq:localTimeError} yields the claim.

\subsection{Proofs of optimal estimation results}

We first present some results on conditional expectations. 

\begin{lemma}\label{lem:condExp}
Suppose that $y\in\R$ and $Z\overset{d}{=}X_1$ is independent of $X$.
\begin{enumerate}
\item $\norm{\c O_T(y)-\E[\c O_{T}(y)|\c G_n]}_{L^{2}(\P)}^{2} 
		 = \Delta_n^2 \sum_{k=0}^{n-1} \E\left[\operatorname{Var}\left(\l{\c O_1(k^{1/\alpha}Z+\Delta_n^{-1/\alpha}y)}  Z,  X_1 \right)\right]$.
\item  If $1<\alpha\leq 2$, then with $C(\alpha)=-(\alpha-1)\Gamma(\alpha)\cos(\frac{\pi \alpha}{2})$
\begin{align*}
   & \norm{L_{T}(y)-\E[\l{L_{T}(y)}\c G_{n}]}_{L^{2}(\P)}^{2}\\
  	& \quad	= \Delta_n^{2-2/\alpha}C(\alpha)^{-2} \sum_{k=0}^{n-1} \E \left[ \operatorname{Var} \left( \l{L_{1}(k^{1/\alpha}Z+\Delta_n^{-1/\alpha}y )} Z,X_{1} \right) \right].
\end{align*}
\end{enumerate}
\end{lemma}
\begin{proof}
(i). Set $g(x)=\I(x\geq y)$, $x\in\R$, such that 
\[
	\c O_T(y)-\E[\c O_{T}(y)|\c G_n]= \sum_{k=1}^n \int_{t_{k-1}}^{t_k}(g(X_{r})-\E[g(X_{r})|\c G_{n}])dr.
\]
By independence and stationarity of the
increments, the Markov property implies for $t_{k-1}\leq r\leq t_k$ 
\begin{align}
\E\left[\l{g(X_{r})}\c G_{n}\right] 
	& =\E\left[\l{g\left(X_{r}\right)}X_{t_{k-1}},X_{t_k}-X_{t_{k-1}}\right]\nonumber \\ 
	& = \E\left[\l{g\left(X_{r}-X_{t_{k-1}} + X_{t_{k-1}}\right)}X_{t_{k-1}},X_{t_k}-X_{t_{k-1}}\right] \label{eq:occFormula} \\
	& \overset{d}{=} \E\left[\l{g\left(\Delta_n^{1/\alpha}X_{(\Delta_n)^{-1}(r-t_{k-1})}+t_{k-1}^{1/\alpha}Z\right)} Z,X_{1}\right],\nonumber
\end{align}
concluding by $X_r-X_{t_{k-1}}\overset{d}{=}\Delta_{n}^{1/\alpha}X_{(\Delta_n)^{-1}(r-t_{k-1})}$, $X_{t_{k-1}}\overset{d}{=}t_{k-1}^{1/\alpha}Z$ due to self-similarity. In particular, the random variables $\int_{t_{k-1}}^{t_k}(g(X_{r})-\E[g(X_{r})|\c G_{n}])dr$ being uncorrelated for different $k$,  
we obtain
\begin{align}
 & \norm{\c O_T(y)-\E[\c O_{T}(y)|\c G_n]}_{L^{2}(\P)}^{2} 
 	 = \sum_{k=1}^{n}\E\left[\text{Var}\left(\l{\int_{t_{k-1}}^{t_k}g(X_{r})dr}  X_{t_{k-1}},X_{t_k}-X_{t_{k-1}}\right)\right]\nonumber\\
 	& \quad = \Delta_n^2 \sum_{k=1}^{n} \E\left[\text{Var}\left(\l{\int_{0}^{1}g(\Delta_n^{1/\alpha}X_{r} - t_{k-1}^{1/\alpha}Z)dr} Z, X_1\right)\right]\nonumber.
\end{align}
The result follows from $g(\Delta_n^{1/\alpha}X_{r}-t_{k-1}^{1/\alpha}Z)=\I(X_{r}\geq (k-1)^{1/\alpha}Z+\Delta_n^{-1/\alpha}y)$.

(ii). Applying the Itô-Tanaka formulas of \cite{protter2005stochastic} for $\alpha=2$ and of \cite[Theorem 2.1]{engelbertkurenok} for $1<\alpha<2$ we find 
\[ 
C(\alpha) L_{T}(y) =|X_{T}-y|^{\alpha-1}-|y|^{\alpha-1}-M_T(y)
\]
for a square integrable martingale $(M_t(y))_t$ such that for all $0<s<t,$ $M_{t}(y)-M_{s}(y)$ is independent of $\sigma(M_u(y),u \leq s)$. Indeed, for $\alpha=2$ we have $M_t(y)=\int_{0}^{t}\text{sgn}(X_{r}-y)dX_{r}$ and for $1<\alpha<2$
\[
	M_t(y) = \int_{0}^{t} \int_{\mathbb{R}} \left[ |X_{u^-}-y+x |^{\alpha-1} - |X_{u^-}-y |^{\alpha-1}\right] q(du,dx),
\]
with $q$ being the compensated Poisson random measure associated with $X$. As in (i), 
\begin{align*}
	\E[\l{M_{k\Delta_n}(y)- M_{(k-1)\Delta_{n}}(y)}\c G_{n}]
		& = \E[\l{M_{k\Delta_n}(y)- M_{(k-1)\Delta_{n}}(y)} X_{t_k-1},X_{t_k}-X_{t_{k-1}}]\\
		& \overset{d}{=}\E[\l{M_{\Delta_n}(y-t_{k-1}^{1/\alpha}Z)} Z,X_{\Delta_n}],
\end{align*}
and the random variables
$(M_{k\Delta_n}(y)- M_{(k-1)\Delta_{n}}(y) -\E[M_{k\Delta_n}(y)- M_{(k-1)\Delta_{n}}(y)|\c G_{n}])_{k}$
are uncorrelated. Consequently, 
\begin{align*}
 & C(\alpha)^2\norm{L_{T}(y)-\E[\l{L_{T}(y)}\c G_{n}]}_{L^{2}(\P)}^{2}
  =\norm{M_T(y)-\E[\l{M_T(x)}\c G_{n}]}_{L^{2}(\P)}^{2}\\
 & \quad =\sum_{k=1}^{n} \E \left[ \operatorname{Var} \left( \l{M_{\Delta_n}(y-t_{k-1}^{1/\alpha}Z)}Z, X_{\Delta_{n}} \right) \right] \\
 & \quad =\sum_{k=0}^{n-1} \E \left[ \operatorname{Var} \left( \l{L_{\Delta_n}(t_{k}^{1/\alpha}Z+y)} Z,X_{\Delta_{n}} \right) \right],
\end{align*}
using the symmetry $Z\overset{d}{=}-Z$ in the last line. The result follows from self-similarity of $X$ and Proposition 10.4.8 of \cite{samorodnitskyStochasticProcessesLong2016}, which implies that $\Delta_n^{1/\alpha-1}L_{\Delta_nt}(\Delta_n^{1/\alpha}a)$ is a version of the local time $L_{t}(a)$. 
\end{proof}

\begin{proof}[Proof of Theorem \ref{Theo2}]
Observe the equality in Lemma \ref{lem:condExp}(i). Since the first summand in that sum is of order $O(\Delta_n^2)$ and thus asymptotically negligible, we only have to study the sum for $k\geq 1$, which equals
\begin{align}
 	\Delta_n^2 \sum_{k=1}^{n-1} \int_{\mathbb{R}} \E\left[\text{Var}\left(\l{\c{O}_1\left(k^{1/\alpha}x+\Delta_n^{-1/\alpha}y\right) }X_1\right)\right] f_{\alpha}(x)dx.\label{eq:mainTerm}
\end{align}
Let first $1<\alpha\leq 2$. After a change of variables  the last line equals
\begin{align*}
 	\int_{\mathbb{R}} \E\left[\text{Var}\left(\l{\c{O}_1\left(x\right) }X_1\right)\right] \Delta_n^{2+1/\alpha} \sum_{k=1}^{n-1} t_{k}^{-1/\alpha}f_{\alpha}(\Delta_n^{1/\alpha}t_{k}^{-1/\alpha}x -t_k^{-1/\alpha}y)dx.
\end{align*}
Since $f_{\alpha}$ is uniformly bounded, dominated convergence implies
\begin{align*}
 	\Delta_n \sum_{k=1}^{n-1} t_{k}^{-1/\alpha}f_{\alpha}(\Delta_n^{1/\alpha}t_{k}^{-1/\alpha}x -t_k^{-1/\alpha}y)\rightarrow \int_0^T t^{-1/\alpha}f_{\alpha}(t^{-1/\alpha}y)dt=\int_0^T f_{\alpha,t}(y)dt.
\end{align*}
This equals $\E[L_T(y)]$ by the occupation time formula \eqref{occfor}, and the claim follows from using dominated convergence again.

Let now $0<\alpha<1$ and suppose $y=0$. Decompose the expected value in \eqref{eq:mainTerm} as $\E[ \text{Var}(Y|X_1)] = \E[Y^2-\E[Y|X_1]^2 ]$ with $Y=\c{O}_1(k^{1/\alpha}x)=\int_0^1\I(X_r\geq k^{1/\alpha}x)dr$. Multiplying out the square  yields
\begin{align*}
\sum_{k=1}^{n-1}\c{O}_1(k^{1/\alpha}x)^2
& = \int_0^1\int_0^1 \sum_{k=1}^{n-1}\I\left(k\leq \frac{|X_r\wedge X_s|^{\alpha}}{x^{\alpha}},X_r\wedge X_s\geq 0\right)drds\\
&= \int_0^1\int_0^1 \left(\left\lfloor\frac{|X_r\wedge X_s|^{\alpha}}{x^{\alpha}}\right\rfloor\wedge(n-1)\right) \I(X_r\wedge X_s\geq 0)drds.
\end{align*}
Writing similarly $\E[\c{O}_1(k^{1/\alpha}x)|X_1]=\int_0^1 \int_{\R}\I(z \geq k^{1/\alpha}x) f_{\alpha,r|X_1}(z)dz dr$ with the conditional density $f_{\alpha,s|X_1}$ from Lemma  \ref{lem:min}, using symmetry and taking expectations allows to rewrite \eqref{eq:mainTerm} as
\begin{align}
	& 2 \Delta_n^2 \int_0^\infty \int_{1}^{\infty} \frac{\left\lfloor z^{\alpha}\right\rfloor\wedge(n-1)}{z^{1+\alpha}}
		   \frac{f_{\alpha}(x)}{x^{\alpha}}h(xz)dzdx,\label{eq:mainExpr}
\end{align}
with $h(u)=\int_0^1\int_0^1u^{1+\alpha}(\rho_{r,s}(u)-\E[\rho_{r,s,X_1}(u)])drds$ for $\rho_{r,s}$, $\rho_{r,s,X_1}$ from Lemma \ref{lem:min}. In order to conclude, consider the decomposition
\begin{align*}
\int_{1}^{\infty} \frac{\left\lfloor z^{\alpha}\right\rfloor\wedge(n-1)}{z^{1+\alpha}}h(xz)dz=\int_{1}^{(n-1)^{1/\alpha}} \frac{\left\lfloor z^{\alpha}\right\rfloor }{z^{1+\alpha}}h(xz) dz+(n-1)\int_{(n-1)^{1/\alpha}}^{\infty} \frac{h(xz)}{z^{1+\alpha}}dz.
\end{align*}
By Lemma \ref{lem:min}(i,ii), $h(u) \rightarrow \frac{h_{\alpha}(\infty)}{12}$, $u\rightarrow \infty$ and $\sup_{u\geq 0} |h(u)|<\infty$, and so the expression in the last display converges uniformly in $x \geq 0$, when divided by $\log n$, to $h_{\alpha}(\infty)/(12\alpha)$. The result is obtained from dominated convergence and noting $2\int_0^\infty \frac{f_{\alpha}(x)}{x^{\alpha}}dx=\E[|Z|^{-\alpha}]$.

At last, let $\alpha=1$ and suppose again $y=0$. As for $0<\alpha<1$ it suffices to study \eqref{eq:mainExpr}, which we rewrite as
\begin{align}
	& 2 \Delta_n^2  \int_{1}^{\infty} \log(z)\frac{\left\lfloor z\right\rfloor\wedge(n-1)}{\pi^2z^2} \bar{h}(z)dz\nonumber\\
	& \quad = 2 \Delta_n^2  \int_{1}^{n-1} \log(z)\frac{\lfloor z\rfloor}{\pi^2z^2} \bar{h}(z)dz + 2 \Delta_n^2(n-1)  \int_{n-1}^{\infty} \log(z)\frac{1}{\pi^2z^2} \bar{h}(z)dz,\nonumber \\
	& \quad \text{with}\quad\bar{h}(z)=\int_0^1\int_0^1 \frac{\pi^2 z^2}{\log(z)}\int_0^\infty x f_{1}(x)(\rho_{r,s}(xz)-\E[\rho_{r,s,X_1}(xz)])dx drds.\label{eq:error_3}
\end{align}
Lemma \ref{lem:min}(iii,iv) yield $\bar{h}(u) \rightarrow \frac{1}{12}$, $u\rightarrow \infty$, $\sup_{u\geq 0} |\bar{h}(u)|<\infty$. Noting $\int_1^{n}\frac{\log(z)}{z}dz=\frac{\log(n)^2}{2}$, the claim follows from dominated convergence.
\end{proof}

\begin{proof}[Proof of Theorem \ref{theorem:Theo4}]
According to Lemma \ref{lem:condExp}(ii) and changing variables we have
\begin{align*}
	& \norm{L_{T}(y)-\E[\l{L_{T}(y)}\c G_{n}]}_{L^{2}(\P)}^{2}\\
  	&	\quad = \Delta_n^{2-1/\alpha}C(\alpha)^{-2} \int_{\R} \E \left[ \operatorname{Var} \left( \l{L_{1}(x)} X_{1} \right) \right]  \sum_{k=1}^{n-1}\frac{1}{t_{k}^{1/\alpha}}f_{\alpha}(\Delta_n^{1/\alpha}t_{k}^{-1/\alpha}x -t_k^{-1/\alpha}y)dx.
\end{align*}
By a typical moment bound for local times, cf. \cite[Lemma 5.2]{marcusrosen2008}, and the same Riemann sum convergence as after \eqref{eq:mainTerm}, we conclude by dominated convergence
\begin{align*}
 & \Delta_{n}^{1/\alpha-1}\norm{L_{T}(y)-\E[\l{L_{T}(y)}\c G_{n}]}_{L^{2}(\P)}^{2} \rightarrow C(\alpha)^{-2}\E[L_T(y)]\int_{\R}\E[\text{Var}(\l{L_1(x)}X_{1})]dx.\qedhere
\end{align*}
\end{proof}

\subsection{Auxiliary lemmas for asymptotic results}

In this section, we collect a number of technical properties for $\alpha$-stable distributions. We begin by recalling a well-known asymptotic property in the case $\alpha \leq 1$. For a proof see \cite[Theorem 2.4.2]{ibragimov}, or \cite[page 246]{ST} for the value of the limit. 

\begin{lemma}\label{lem:densityProps}
If $0<\alpha \leq 1$, then $h_{\alpha}(x) = x^{1+\alpha} f_{\alpha}(x)$, $x\geq 0$, is a non-negative non-decreasing function satisfying 
\begin{equation}
 \lim_{x \rightarrow +\infty}\limits h_{\alpha}(x) =  \frac{\alpha}{\pi} \sin(\frac{\pi \alpha}{2} )\Gamma(\alpha)=:h_{\alpha}(\infty).
\end{equation}

\end{lemma}

The next two lemmas are used in Theorems \ref{Theo2} and \ref{Theo1}. 

\begin{lemma}\label{lem:min}
For $0<\alpha\leq 1$ let $\rho_{r,s}$ denote the Lebesgue density of $X_r\wedge X_s$ for $0<r,s\leq 1$, $f_{\alpha,s|X_1}$ is the conditional Lebesgue density of $X_s$ for $0<s< 1$, conditional on $X_1$ and set  $\rho_{r,s,X_1}(z)=2f_{\alpha,r|X_1}(z) \int_z^{\infty} f_{\alpha,s|X_1}(\tilde{z})d\tilde{z}$, $z\geq 0$. Then, as $z\rightarrow \infty$:
\begin{enumerate}
\item $z^{1+\alpha}\rho_{r,s}(z)\rightarrow (r\wedge s) h_{\alpha}(\infty)$, 
\item $z^{1+\alpha}\E[\rho_{r,s,X_1}(z)]\rightarrow rsh_{\alpha}(\infty)$.
\item $\frac{\pi^2(z/r)^2}{\log(z/r)} \int_0^\infty x f_{1}(x)\rho_{r,s}(xz)dx	\rightarrow \frac{1}{r\wedge s}$ if $\alpha=1$.
\item $\frac{\pi^2(z/r)^2}{\log(z/r)}\int_0^\infty x f_{1}(x)\E[\rho_{r,s,X_1}(xz)]dx
	  \rightarrow \frac{r\vee s}{r\wedge s}$ if $\alpha=1$.
\end{enumerate}
\end{lemma}
\begin{proof}
By independence of increments it is not difficult to see for $0<r<s$ that 
\begin{align}\label{expressionrho}
\rho_{r,s}(u) & =\frac{1}{2}f_{\alpha,r}(u) + \int_0^{\infty}f_{\alpha,r}(u+v)f_{\alpha,s-r}(v) dv.
\end{align}
Monotone convergence and Lemma \ref{lem:densityProps} yield (i). On the other hand, again by independence of increments, the conditional density for $0<r<1$ is given by
\[
	f_{\alpha,r|X_1=x}(z) = \frac{f_{\alpha,r}(z)f_{\alpha,1-r}(x-z)}{f_{\alpha,1}(x)},
\]
recalling that $f_{\alpha,1}=f_{\alpha}$. Denote $\Psi_r(z,x)=\int_{z}^{\infty}f_{\alpha,r|X_1=z+x}(\tilde{z})d\tilde{z}$. As above, monotone convergence and Lemma \ref{lem:densityProps} imply then as $z\rightarrow\infty$
\begin{align*}
	\Psi_r(z,x)
		& = \frac{1}{(z+x)^{1+\alpha}f_{\alpha,1}(z+x)}\int_{0}^{\infty}(z+x)^{1+\alpha}f_{\alpha,r}(z+\tilde{z})f_{\alpha,1-r}(x-\tilde{z}) d\tilde{z}\\
		& \rightarrow  \frac{1}{h_{\alpha}(\infty)}\int_{0}^{\infty}rh_{\alpha}(\infty)f_{\alpha,1-r}(x-\tilde{z}) d\tilde{z}=r\int_{-x}^{\infty}f_{\alpha,1-r}(\tilde{z})d\tilde{z},
\end{align*}
using symmetry of $f_{\alpha,1-r}$ in the last line. Since also trivially $\Psi(z,x)\leq \Psi(-\infty,x)=1$, we find from this by dominated convergence 
\begin{align}
	z^{1+\alpha}\E[\rho_{r,s,X_1}(z)]
		& = 2z^{1+\alpha}\int_\R \frac{f_{\alpha,r}(z)f_{\alpha,1-r}(x-z)}{f_{\alpha,1}(x)} 
		\int_z^{\infty}f_{\alpha,s|X_1=x}(\tilde{z})d\tilde{z} f_{\alpha,1}(x)dx\nonumber \\
		& = 2z^{1+\alpha}f_{\alpha,r}(z) \int_\R f_{\alpha,1-r}(x)
		\Psi_s(z,x) dx\label{eq:expressionrhoCond} \\
		& \rightarrow 2rsh_{\alpha}(\infty)\int_\R f_{\alpha,1-r}(x)\int_{-x}^{\infty}f_{\alpha,1-s}(\tilde{z})d\tilde{z}dx=rsh_{\alpha}(\infty),\nonumber
\end{align}
again using symmetry of the densities. This shows (ii). For (iii) let again $0<r<s$. From \eqref{expressionrho} we find that
\begin{align*}
	\int_0^\infty x f_{1}(x)\rho_{r,s}(xz)dx=\frac{1}{2r}g_1(0,\frac{z}{r}) + \frac{1}{r}\int_0^\infty g_1(\frac{v}{r},\frac{z}{r})f_{\alpha,s-r}(v)dv,
\end{align*}
where $g_1(a,b) = \int_0^{\infty} x f_1(x) f_1(a+bx)dx$. Lemma \ref{lemG1} implies for $z\rightarrow \infty$
\begin{align*}
	\frac{\pi^2(z/r)^2}{\log(z/r)} \int_0^\infty x f_{1}(x)\rho_{r,s}(xz)dx
		\rightarrow \frac{1}{2r} + \frac{1}{r}\int_0^\infty f_{1,s-r}(v)dv=\frac{1}{r}.
\end{align*}
At last, starting from \eqref{eq:expressionrhoCond}, we have
\begin{align*}
	\int_0^\infty x f_{1}(x)\E[\rho_{r,s,X_1}(xz)]dx
	  =2\int_0^\infty xf_1(x) f_{\alpha,r}(xz) \int_\R f_{\alpha,1-r}(\tilde{x})
	\Psi_s(xz,\tilde{x}) d\tilde{x}dx.
\end{align*}
For $\alpha=1$, the density is $f_1(x)=\frac{1}{\pi(1+x^2)}$  and so $\frac{(z/r)^2}{\log(z/r)}xf_1(x)f_{\alpha,r}(xz)$ is clearly integrable uniformly in $z\geq 0$. Since $\Psi_s(xz,\tilde{z})\leq 1$ by (ii) and $\Psi_s(xz,\tilde{x})\rightarrow s\int_{-x}^{\infty}f_{\alpha,1-s}(\tilde{z})d\tilde{z}$, the last display equals, as $z\rightarrow \infty$,
\begin{align*}
& 2s\int_0^\infty xf_1(x) f_{\alpha,r}(xz) \int_\R f_{\alpha,1-r}(\tilde{x})\int_{-\tilde{x}}^{\infty}f_{\alpha,1-s}(\tilde{z})d\tilde{z}dx+ o\left(\frac{\log(z/r)}{(z/r)^2}\right)\\
&\quad =\frac{s}{r}g_1(0,\frac{z}{r}) + o\left(\frac{\log(z/r)}{(z/r)^2}\right).
\end{align*}
We conclude with Lemma  \ref{lemG1}.
\end{proof}

\begin{lemma}\label{lem:lem4} Let $f_{D}$ denote the Lebesgue density of $D=(1+|\tilde{Z}/Z|^{\alpha})^{-1}$.
\begin{enumerate}
\item If $1<\alpha\leq 2$, then $\sup_{0<x<1/2}|x^{1-1/\alpha}f_{D}(x)|<\infty$ and $x^{1-1/\alpha}f_{D}(x)\rightarrow \frac{2}{\pi \alpha^2}\Gamma(\frac{1}{\alpha})\E[|Z|]$ as $x\rightarrow 0$.
\item If $0<\alpha<1$, then $\sup_{0<x<1/2}f_{D}(x)\leq 4$ and $f_{D}(x)\rightarrow \frac{2}{\pi}\sin(\frac{\pi \alpha}{2} )\Gamma(\alpha)\E[|Z|^{-\alpha}]$ as $x\rightarrow 0$.
\item If $\alpha=1$, then $f_{D}(x)=\frac{4}{\pi^2}(1-2x)^{-1}\log(x^{-1}-1)$.
\end{enumerate}
\end{lemma}

\begin{proof}
For all $0<\alpha\leq 2$ denote by $g_{\alpha}(x) =4 \int_0^{\infty} y f_{\alpha}(y) f_{\alpha}(xy)dy$, $x > 0$, the density of $|\tilde{Z}/Z|$, such that
\begin{equation}
f_{D}(x) = \frac{1}{\alpha} x^{-1-1/\alpha} (1-x)^{-1+1/\alpha}g_{\alpha}((x^{-1}-1)^{1/\alpha}),\quad 0<x<1.\label{eq:fdelta}
\end{equation} 
(i). For $1<\alpha\leq 2$, $\E[|Z|]$ is finite and therefore $g_{\alpha}(x)\leq 2 \sup_{y>0}f_{\alpha}(y) \E[|Z|]$. Together with the property $g_{\alpha}(x)=x^{-2}g_{\alpha}(x^{-1})$ for $x\neq 0$, we get from \eqref{eq:fdelta} for $0<x<1/2$ that
\begin{align*}
|x^{1-1/\alpha}f_{D}(x)| 
& \lesssim x^{-2/\alpha} (1-x)^{-1+1/\alpha}(x^{-1}-1)^{-2/\alpha} = (1-x)^{-1-1/\alpha} \lesssim 1.
\end{align*}
The claimed convergence follows from $g_{\alpha}(x)\rightarrow 2f_{\alpha}(0)\E[|Z|]$ as $x\rightarrow 0$ and from $f_{\alpha}(0)=(\alpha \pi)^{-1}\Gamma(1/ \alpha)$.

(ii). Observe first that
\[
x^{1+\alpha}g_{\alpha}(x)=4\int_0^\infty \frac{f_{\alpha}(y)}{y^{\alpha}}h_{\alpha}(xy)dy
\]
is non-decreasing in $x>0$ and converges to $2h_{\alpha}(\infty)\E[|Z|^{-\alpha}]$ as $x\rightarrow \infty$ by the monotone convergence theorem and Lemma \ref{lem:densityProps}. This yields with \eqref{eq:fdelta} for $0<x<1/2$
\begin{align*}
f_{D}(x)
& \lesssim x^{-1-1/\alpha}(1-x)^{-1+1/\alpha}(x^{-1}-1)^{-1-1/\alpha} = (1-x)^{-2} < 4,
\end{align*}
and the claimed convergence of $f_{D}(x)$ as $x\rightarrow 0$. 

(iii). For $\alpha=1,$ the equality is a straightforward consequence of Lemma \ref{lemG1} and the equality $f_{D}(x)=4x^{-2}g_1(0,x^{-1}-1).$
\end{proof}

\begin{lemma}\label{lem:lem3}
Recall the function $\psi(x) = (x-\lfloor x\rfloor)-(x-\lfloor x\rfloor)^{2}$, $x\geq 0$, from Theorem \ref{Theo1}. 
It holds:
 \begin{enumerate}
 \item $\lim_{n \rightarrow +\infty}\limits \frac{1}{\log n} \int_0^{n/2} \frac{\psi(x)}{x(1-x/n)}dx = \frac{1}{6}$,
 \item $\lim_{n \rightarrow +\infty}\limits \frac{1}{(\log n)^2} \int_0^{n/2} \frac{\psi(x)\log ((x/n)^{-1}-1)}{x(1-x/n)(1-2x/n)}dx = \frac{1}{12}$.
\end{enumerate}
\end{lemma}

\begin{proof}
(i). Using the relationship $x^{-1}(1-x/n)^{-1} - 1 = n^{-1}(1-x/n)^{-1}\leq 2n^{-1}$ for $0<x<n/2$ and $\psi(x)\leq 1$, it is enough to study the limit of $\frac{1}{\log n} \int_0^{n/2} \frac{\psi(x)}{x}dx$. Moreover, since we have for $n=2k+1$ odd the bound $\int_{k}^{k+1/2} \frac{\psi(x)}{x}dx\leq 1/k \lesssim 1/n$, we only have to consider $n=2k$ even.  
Denoting by $[x]=x-\lfloor x\rfloor$ the fractional part of $x\in\R$, we have
\begin{equation*}
\int_0^{k} \frac{\psi(x)}{x} dx
	= \sum_{j=0}^{k-1} \int_j^{j+1} \frac{[x]-[x]^2}{j+[x]}dx
	= \sum_{j=0}^{k-1} \int_0^{1} \frac{x-x^2}{j+x}dx.
 \end{equation*}
Since $j^{-1}-(j+x)^{-1}\leq j^{-2}$ for $0<x<1$ is summable in $j$, we conclude that
\[
\lim_{k \rightarrow +\infty}\limits \frac{1}{\log 2k} \int_0^{k} \frac{\psi(x)}{x}dx 
	= \lim_{k \rightarrow +\infty}\limits \frac{1}{\log 2k}  \sum_{j=0}^{k-1} \frac{1}{j}\int_0^{1} (x-x^2)dx
	= \frac{1}{6}.
\]

(ii). Note that 
\[
\frac{1}{x(1-x/n)(1-2x/n)} = \frac{1}{x}+\frac{2}{n(1-2x/n)}+\frac{1}{n(1-x/n)(1-2x/n)}.
\]
For $0<x<n/2$ the last term is bounded by the second one. Since
\[
\int_0^{n/2} \frac{\psi(x)\log ((x/n)^{-1}-1)}{n(1-2x/n)}dx\leq \int_0^{1/2} \frac{\log (x^{-1}-1)}{1-2x}dx<\infty,
\]
it is thus enough to compute the limit of
\begin{align*}
& \lim_{n \rightarrow +\infty}\frac{1}{(\log n)^2} \int_0^{n/2} \frac{\psi(x)\log (\frac{n-x}{x})}{x}dx \\
	& \quad = \lim_{n \rightarrow +\infty} \frac{1}{\log n} \int_0^{n/2} \frac{\psi(x)}{x}dx 
		- \lim_{n \rightarrow +\infty} \frac{1}{(\log n)^2} \int_0^{n/2} \frac{\psi(x)\log x}{x}dx,
\end{align*}
where we used $\log((x/n)^-1-1)=\log n + \log(1-x/n)-\log x$ in the last line. We show below that the second term equals $-1/12$. Together with the result from (i) we get (ii). In order to find the limit for the second term, note that as in (i) it is enough to consider $n=2k$ even. As above, we have
\begin{equation*}
\int_0^{k} \frac{\psi(x)\log x}{x} dx
	= \sum_{j=0}^{k-1} \int_0^{1} \frac{(x-x^2)\log(j+x)}{j+x}dx.
\end{equation*}
Noting that $|\frac{\log(j+x)}{j+x}-\frac{\log j}{j}|\leq j^{-2}$ for $0<x<1$ is summable, the claimed limit is obtained from 
\[
\lim_{k \rightarrow +\infty}\limits \frac{1}{(\log 2k)^2} \int_0^{k} \frac{\psi(x)\log x}{x}dx 
	= \lim_{k \rightarrow +\infty}\limits \frac{1}{(\log 2k)^2}  \sum_{j=0}^{k-1} \frac{\log j}{6j}
	= \frac{1}{12}. \qedhere
\]
\end{proof}

\begin{lemma}\label{lemG1}
Let $\alpha=1$ and define  
  \begin{equation*}
   g_1(a,b) = \int_0^{\infty} x f_1(x) f_1(a+bx) dx = \frac{1}{\pi^2} \int_0^{\infty} \frac{x}{(1+x^2)(1+(a+bx)^2)} dx.
  \end{equation*}
Then for every $a \geq0$ and $b>0,$ we get the following equality
 \begin{align}
  g_1(a,b)  & = \frac{1}{\pi^2 ((1+a^2-b^2)^2+4a^2b^2 )} \big[\frac{1+a^2-b^2}{2} \log(\frac{1+a^2}{b^2} ) \nonumber \\
  & +\pi ab-a(1+a^2-b^2 )(\frac{\pi}{2} - \arctan(a)) \big]. \label{G1tot}
 \end{align}
 In particular,  
 \begin{equation*}
  g_1(0,b) = \frac{\log b}{\pi^2(b^2-1)},\quad  \lim_{b \rightarrow +\infty} \frac{\pi^2 b^2}{\log b} g_1(a,b) = 1,
\end{equation*} 
  and  $g_1(a,b) \leq g_1(0,b) = \frac{\log b}{\pi^2(b^2-1)}$.

\end{lemma}

\begin{proof} 
We only compute $\pi^2 g_1(a,b) = \int_0^{\infty} \frac{x}{(1+x^2)(1+(a+bx)^2)}dx,$ all the other properties being easily deduced from \eqref{G1tot}. Writing $K_{a,b}^{-1} = (1+a^2-b^2)^2+4a^2b^2,$ we shall use the decomposition
\[ \frac{K_{a,b}^{-1}}{(1+x^2)(1+(a+bx)^2)} = \frac{1+a^2-b^2-2abx}{1+x^2} + \frac{b^2( 3a^2+b^2-1) + 2ab^3x}{1+(a+bx)^2}\]
to obtain 
\[\frac{x K_{a,b}^{-1}}{(1+x^2)(1+(a+bx)^2)} = \frac{1+a^2-b^2}{2} \left[ \frac{2x}{1+x^2}- \frac{2xb^2+2ab}{1+(a+bx)^2} \right] + \frac{2ab}{1+x^2} - \frac{ab(1+a^2+b^2)}{1+(a+bx)^2}\]
and
\begin{align*}
\frac{\pi^2g_1(a,b)}{K_{a,b}} &= \lim_{t \rightarrow +\infty} \frac{1+a^2-b^2}{2} \left[\log(1+t^2) - \log(1+(a+bt)^2) + \log(1+a^2)\right] \\&+  \lim_{t \rightarrow +\infty}  2ab\arctan(t)- a(1+a^2+b^2) (\arctan(a+bt)-\arctan(a)) 
\end{align*}
which yields \eqref{G1tot}.
%

\end{proof}

\appendix

\bibliographystyle{apalike}
\bibliography{references}

\end{document}